\documentclass[a4paper,11pt]{article}
\usepackage[top=3cm, bottom=3cm, left=3.5cm, right=3.5cm]{geometry}
\usepackage[T1]{fontenc}
\usepackage{amsmath,amssymb,amsthm}
\usepackage{enumitem}
\usepackage{placeins}
\usepackage{tikz}
\usetikzlibrary{arrows.meta,decorations.pathreplacing,positioning,patterns}

\tikzset{
  vtx/.style={circle,fill=white,draw=black,
              minimum size=4.4pt,inner sep=0pt},
  cell/.style={draw=black,line width=.6pt},
  accent/.style={fill=blue!12},
  forbidden/.style={pattern=north east lines,pattern color=black!45},
  hook/.style={line width=1pt,densely dotted},
}

\newcommand{\hpad}[1]{%
  \raisebox{0pt}[\dimexpr\height+8pt\relax][\dimexpr\depth+4pt\relax]{#1}}

\newcommand{\hassea}[1]{\makebox[14mm]{\hpad{%
  \begin{tikzpicture}[baseline=(current bounding box.south),
      x=3.4mm,y=4.6mm,semithick]
    \foreach \i/\x/\y in {#1} \node[vtx] (n\i) at (\x,\y) {};
  \end{tikzpicture}}}}
\newcommand{\hasse}[2]{\makebox[14mm]{\hpad{%
  \begin{tikzpicture}[baseline=(current bounding box.south),
      x=3.4mm,y=4.6mm,semithick]
    \foreach \i/\x/\y in {#1} \node[vtx] (n\i) at (\x,\y) {};
    \foreach \a/\b in {#2} \draw (n\a)--(n\b);
  \end{tikzpicture}}}}

\newtheorem{theorem}{Theorem}[section]
\newtheorem{lemma}[theorem]{Lemma}
\newtheorem{corollary}[theorem]{Corollary}
\newtheorem{proposition}[theorem]{Proposition}
\theoremstyle{definition}
\newtheorem{definition}[theorem]{Definition}
\newtheorem{example}[theorem]{Example}
\theoremstyle{remark}

\newtheorem*{remark*}{Remark}

\newcommand{\Virt}{\mathrm{Virt}}
\newcommand{\Spe}{\mathrm{Spe}}
\newcommand{\CM}[1]{\mathrm{CM}^{\langle #1\rangle}}
\newcommand{\CMsum}{\mathrm{CM}}
\newcommand{\BM}{\mathrm{BalMat}}
\newcommand{\cB}{\mathcal{B}}
\newcommand{\Int}{\mathrm{Int}}
\newcommand{\Aut}{\operatorname{Aut}}
\newcommand{\fix}{\operatorname{fix}}
\newcommand{\Fix}{\operatorname{Fix}}
\newcommand{\Row}{\operatorname{row}}
\newcommand{\Col}{\operatorname{col}}
\newcommand{\Sym}{\mathcal{S}}
\newcommand{\SymOn}{\operatorname{Sym}}
\newcommand{\I}{\mathcal{I}}
\newcommand{\Cyc}{\mathcal{C}}
\newcommand{\cT}{\mathcal{T}}
\newcommand{\cR}{\mathcal{R}}
\newcommand{\BB}{\mathbb{B}}
\newcommand{\NN}{\mathbb{N}}
\newcommand{\ZZ}{\mathbb{Z}}

\begin{document}

\title{The species of interval orders}

\author{Anders Claesson\\[2pt]
  \small Department of Mathematics, University of Iceland\\
  \small\texttt{akc@hi.is}}
\date{22 August 2026}

\maketitle
\thispagestyle{empty}

\begin{abstract}
  We show that, in the ring of virtual species,
  \[
    \I \;=\; \sum_{m\geq 0}\,(-1)^m
    \prod_{i=1}^{m}\bigl((E^{-1})^i-1\bigr),
  \]
  where $\I$ is the species of interval orders and $E^{-1}$ is the
  multiplicative inverse of the species $E$ of sets. The right-hand side
  is the virtual species of signed ballot matrices introduced by Claesson
  and Hannah. They showed that its signed cardinality counts labeled
  interval orders. We strengthen
  this to a species identity, which we prove twice: first algebraically
  and then bijectively, using a natural sign-reversing involution. The cycle
  index series of $\I$ specializes to the
  generating series for labeled and unlabeled interval orders. We describe
  the automorphism group of an
  interval order as a Young subgroup and prove the identity
  $\I=\cR\circ E_+$, where $\cR$ is the species of rigid interval
  orders. We also show that Glaisher's
  T-number $T_n$ counts the $24$-colored interval orders on $[n]$ in
  which no isolated element has color $24$.
\end{abstract}

{\small

\par}

\section{Introduction}

An \emph{interval order} is a poset whose elements can be represented by
closed real intervals in such a way that $x<y$ if and only if the
interval of $x$ lies entirely to the left of the interval of $y$. By a
theorem of Fishburn~\cite{Fishburn1970}, the finite interval orders
are exactly the $(2{+}2)$-free posets, those with no induced subposet
isomorphic to the disjoint union of two two-element chains.
Bousquet-M\'elou, Claesson, Dukes and Kitaev~\cite{BCDK2010} proved
that the number of unlabeled interval orders on $n$ points is the
coefficient of $x^n$ in
\begin{equation}\label{eq:ogf}
  \sum_{m\geq 0}\,\prod_{i=1}^{m}\bigl(1-(1-x)^i\bigr)
  \;=\; 1+x+2x^2+5x^3+15x^4+53x^5+\cdots,
\end{equation}
a series first studied by Zagier~\cite{Zagier2001}. The coefficients
are called the \emph{Fishburn numbers}~\cite{CL2011} and form sequence
A022493 in the OEIS~\cite{OEIS}. Claesson, Dukes and
Kubitzke~\cite{CDK2011} derived the generating series for labeled interval
orders,
\begin{equation}\label{eq:egf}
  \sum_{m\geq 0}\,\prod_{i=1}^{m}\bigl(1-e^{-ix}\bigr)
  \;=\; 1+x+3\frac{x^2}{2!}+19\frac{x^3}{3!}+207\frac{x^4}{4!}+\cdots
\end{equation}
The coefficients form sequence A079144 in the OEIS~\cite{OEIS}.
Claesson and Hannah~\cite{CH2014} recast \eqref{eq:egf} in the
language of species. Here $e^{-x}$ is the generating series of the
virtual species $E^{-1}$, the multiplicative inverse of the species
$E$ of sets. Thus, the left-hand side of \eqref{eq:egf} suggests the
virtual species
\begin{equation}\label{eq:balmat}
  \BM
  \;=\; \sum_{m\geq 0}\,\prod_{i=1}^{m}\bigl(1-(E^{-1})^i\bigr)
  \;=\; \sum_{m\geq 0}\,(-1)^m\prod_{i=1}^{m}\bigl((E^{-1})^i-1\bigr).
\end{equation}
They represent $\BM$ by signed ballot matrices, and use a
sign-reversing involution to show that the signed cardinality on
$[n]=\{1,\dots,n\}$ is the number of labeled interval orders on $[n]$.

Let $\I$ denote the species of interval orders. Thus, $\I[U]$ is the set
of $(2{+}2)$-free posets on the finite set $U$, and a bijection
$\sigma\colon U\to V$ transports a poset by relabeling. We prove that
$\BM=\I$ in the ring of virtual species.

\begin{theorem}\label{thm:intro}
  In the ring of virtual species,
  \[
    \I \;=\; \sum_{m\geq 0}\,(-1)^m
    \prod_{i=1}^{m}\bigl((E^{-1})^i-1\bigr).
  \]
\end{theorem}

Several proofs of \eqref{eq:ogf} use cancellation. Levande
\cite{Levande2013} gave a sign-reversing involution on Fishburn diagrams.
Eriksen and Sj\"ostrand \cite{ES2014} used the inclusion-exclusion principle
to refine the series by statistics on matchings and permutations, and Hannah
\cite{Hannah2015} used it to recover the enumeration of unlabeled interval
orders from factorial posets.

Section~\ref{sec:species} recalls the language of species and virtual
species. We give two proofs of Theorem~\ref{thm:intro} at the level of
species. The algebraic proof
(Sections~\ref{sec:canonical}--\ref{sec:main}) combines a composition matrix
representation of interval orders~\cite{CDK2011,DJK2011} with
inclusion-exclusion over the matrix columns. The bijective proof
(Section~\ref{sec:involution}) constructs a natural sign-reversing involution
on ballot matrices whose fixed points are identified with composition
matrices.

Equations~\eqref{eq:ogf} and~\eqref{eq:egf} are direct consequences of
Theorem~\ref{thm:intro}. An expression for the cycle
index series of $\I$ also follows (Section~\ref{sec:series}):
\[
  Z_{\I}(x_1,x_2,x_3,\dots) \;=\;
  \sum_{m\geq 0}\prod_{i=1}^{m}
  \biggl(1-\exp\Bigl(-i\Bigl(\,x_1+\frac{x_2}{2}+\frac{x_3}{3}
    +\cdots\,\Bigr)\Bigr)\biggr).
\]
In Section~\ref{sec:molecular} we describe the automorphism group of an
interval order and the molecular decomposition of $\I$. In particular,
the number of interval orders on $[n]$ fixed by a permutation with $c$
cycles is the number of labeled interval orders on $[c]$, and so
depends on $c$ alone. In Section~\ref{sec:rigid} we show that the
indistinguishability classes of an interval order form a rigid interval
order. This gives $\I=\cR\circ E_+$, in which $\cR$ is the species of
rigid interval orders.

In Section~\ref{sec:glaisher} we show that Glaisher's
T-numbers count certain colored interval orders.

\section{Virtual species}\label{sec:species}

We follow the notation and terminology of Bergeron,
Labelle and Leroux~\cite{BLL1998}.

\subsection{Species and their series}
A species defines both a class of labeled combinatorial structures and
how those structures are affected by relabeling. More precisely, a
species $F$ associates with each finite set $U$ a finite set $F[U]$,
whose elements are called \emph{$F$-structures} on $U$, and with each
bijection $\sigma\colon U\to V$ a bijection
$F[\sigma]\colon F[U]\to F[V]$, called \emph{transport of structure}. These
data satisfy
\[
  F[\mathrm{id}_U]=\mathrm{id}_{F[U]},
  \qquad
  F[\tau\circ\sigma]=F[\tau]\circ F[\sigma]
\]
for bijections $\sigma\colon U\to V$ and $\tau\colon V\to W$.
Equivalently, a species is a functor from $\BB$, the category of
finite sets and bijections, to the category of finite sets and
functions. We write
$\sigma\cdot s$ for $F[\sigma](s)$. Two structures $s\in F[U]$ and
$t\in F[V]$ are \emph{isomorphic} if $\sigma\cdot s=t$ for some bijection
$\sigma\colon U\to V$. Two species $F$ and $G$ are \emph{combinatorially
equal}, written $F=G$, if there is a natural isomorphism between them;
that is, if there is a family of bijections
$F[U]\to G[U]$ commuting with transport. We use the standard species
$1$ (empty set), $X$
(singletons), $E$ (sets), $E_k$ (sets of cardinality $k$), $E_+$
(nonempty sets), $L$ (linear orders), $\Cyc$ (oriented cycles),
and the sum and product of species. A species $F$ has three associated series.
The \emph{generating series} and the \emph{type generating series} are
\[
  F(x)=\sum_{n\geq 0}\bigl|F[n]\bigr|\,\frac{x^n}{n!}\,,
  \qquad
  \widetilde{F}(x)=\sum_{n\geq 0}\bigl|F[n]/\!\sim\bigr|\;x^n ,
\]
where $\sim$ is isomorphism of structures. The \emph{cycle index series} is
\[
  Z_F(x_1,x_2,x_3,\dots)=\sum_{n\geq 0}\frac{1}{n!}
  \sum_{\sigma\in\Sym_n}
  \fix F[\sigma]\; x_1^{\sigma_1}x_2^{\sigma_2}x_3^{\sigma_3}\cdots,
\]
where $\fix F[\sigma]$ is the number of $F$-structures on $[n]$ fixed
by $\sigma$ and $(\sigma_1,\sigma_2,\dots)$ is the cycle type of
$\sigma$. All three are additive and multiplicative, and
\[
  F(x)=Z_F(x,0,0,\dots),
  \qquad
  \widetilde{F}(x)=Z_F(x,x^2,x^3,\dots).
\]

A nonzero species $M$ is \emph{molecular} if any two
$M$-structures are isomorphic.
Such an $M$ is concentrated on sets of some fixed cardinality $n$,
and transport makes $M[n]$ a transitive $\Sym_n$-set. Choosing
$s\in M[n]$ with stabilizer, or automorphism group,
$H=\{\sigma\in\Sym_n:\sigma\cdot s=s\}$ identifies the $M$-structures
on $[n]$ with the cosets of $H$. In this notation we have
$M= X^n/H$, and
another choice of $s$ replaces $H$ by a conjugate. Familiar
examples are
\[
  X^n=X^n/\{1\},\qquad
  E_n=X^n/\Sym_n,\qquad
  \Cyc_n=X^n/\langle(1\,2\,\cdots\,n)\rangle,
\]
the species of linear orders, of sets, and of oriented cycles on
$n$ elements. On three elements, $XE_2$, a distinguished element
together with an unordered pair, is molecular with $H$ generated
by the transposition exchanging the pair.

Every species is the sum of its molecular subspecies. Indeed, for each
$n$ the orbits of $\Sym_n$ on $F[n]$, one for each isomorphism type of
$F$-structure, are molecular summands. Collecting isomorphic
summands gives the \emph{molecular decomposition}
\[
  F\;=\;\sum_{M}f_M\,M
  \qquad(f_M\in\NN),
\]
the sum running over isomorphism classes of molecular species. For
instance, $E=\sum_n E_n$ and $L=\sum_n X^n$, while the
species of graphs on two-element sets decomposes as $2E_2$, since the
empty graph and the one-edge graph give two summands of the same
molecular type. The decomposition is unique, so $F= G$ if and only if
$f_M=g_M$ for every molecular species $M$.

\subsection{The ring of virtual species}
A \emph{virtual species} is an element of
\[
  \Virt=(\Spe\times\Spe)/\!\sim\,,
  \qquad
  (F,G)\sim(H,K)\iff F+K= G+H,
\]
where $\Spe$ is the semiring of species. The equivalence class of
$(F,G)$ is written $F-G$. With the evident addition and
multiplication, $\Virt$ is a commutative ring, and
$F\mapsto F-0$ embeds $\Spe$ into $\Virt$.
Suppose that $F$ and $G$ have molecular decompositions
\[
  F=\sum_M f_M M,
  \qquad
  G=\sum_M g_M M.
\]
Then their difference has the molecular expansion
\[
  \Phi=F-G=\sum_M \phi_M M,
  \qquad
  \phi_M=f_M-g_M\in\ZZ.
\]
The coefficients $\phi_M$ depend only on $\Phi$. Consequently, every
virtual species has a unique \emph{reduced form}
\[
  \Phi=\Phi^+-\Phi^-,
  \qquad
  \Phi^+=\sum_{\phi_M>0}\phi_M M,
  \qquad
  \Phi^-=\sum_{\phi_M<0}(-\phi_M)M,
\]
in which $\Phi^+$ and $\Phi^-$ have no common molecular summand. We call
$\Phi$ \emph{positive} if $\Phi^-=0$, equivalently, if every $\phi_M$ is
nonnegative. The three associated series extend by differences,
\[
  \Phi(x)=F(x)-G(x),
  \qquad
  \widetilde{\Phi}(x)=\widetilde{F}(x)-\widetilde{G}(x),
  \qquad
  Z_\Phi=Z_F-Z_G,
\]
and remain additive and multiplicative. In particular, if $\Phi$ is
positive then $\Phi=\Phi^+$ is a species, so $\Phi(x)$,
$\widetilde{\Phi}(x)$ and $Z_\Phi$ all have nonnegative coefficients.
These three series do not determine a virtual species, and
nonnegativity of their coefficients does not imply positivity.
Bergeron, Labelle and Leroux give the following
example~\cite[Section~2.6, equations~(30)--(31)]{BLL1998}.
Consider the four molecular species $X^3$, $XE_2$, $\Cyc_3$ and $E_3$.
Their fixed-point counts, one representative per
cycle type, are
\[
  {\renewcommand{\arraystretch}{1.2}
    \begin{array}{c|ccc}
      & \mathrm{id} & (1\,2) & (1\,2\,3)\\ \hline
      X^3  & 6 & 0 & 0\\
      XE_2 & 3 & 1 & 0\\
      \Cyc_3  & 2 & 0 & 2\\
      E_3  & 1 & 1 & 1
    \end{array}
  }
\]
Summing the relevant rows of the table shows that the species
\[
  2\,XE_2+\Cyc_3
  \qquad\text{and}\qquad
  X^3+2\,E_3
\]
have the same number of fixed structures, namely $8$, $2$ and $2$,
for every permutation of $[3]$, and hence equal cycle index series.
Yet their molecular decompositions differ, so they are not combinatorially equal.
Their difference $2\,XE_2+\Cyc_3-X^3-2\,E_3$ is therefore a nonzero
virtual species whose three series all vanish, and it is not positive.
In particular, Theorem~\ref{thm:intro} does not follow from the
preceding counting results~\cite{BCDK2010,CDK2011,CH2014}.

We use infinite sums of (virtual) species in the standard \emph{summable}
sense: on each finite set, only finitely many summands contribute.
For example, the sums in \eqref{eq:balmat} are summable because every
factor $(E^{-1})^i-1$ vanishes on the empty set, so the $m$th product
vanishes on sets of cardinality less than $m$. Such sums may be
regrouped freely. Their generating, type generating, and cycle index
series are obtained by summing the corresponding series of the
summands coefficientwise.

\subsection{Ballot matrices}
The multiplicative inverse
\[
  E^{-1} = (1+E_+)^{-1} = \sum_{k\geq 0} (-1)^kE_+^k
\]
of the set species is a standard example of a virtual species. On a
finite set $U$, an $E^{-1}$-structure is a \emph{ballot}, a sequence of
pairwise disjoint nonempty \emph{blocks} with union $U$. A ballot with
$k$ blocks has sign $(-1)^k$. Thus, the ballots with an even number of
blocks make up $(E^{-1})^{+}$, while those with an odd number make up
$(E^{-1})^{-}$. The three associated series are
\begin{equation}\label{eq:einv-values}
  E^{-1}(x)=e^{-x},
  \qquad
  \widetilde{E^{-1}}(x)=1-x,
  \qquad
  Z_{E^{-1}}=e^{-q},
\end{equation}
where here and throughout
\[
  q \;=\; x_1+\frac{x_2}{2}+\frac{x_3}{3}+\cdots,
  \quad\text{so that}\quad
  Z_E=e^{q}.
\]
Signed ballot matrices model the virtual species $\BM$ of
\eqref{eq:balmat}. We shall use this model in the second proof of
Theorem~\ref{thm:intro}.

\begin{definition}[Claesson and Hannah~\cite{CH2014}]
  A \emph{ballot matrix} on a finite set $U$ is an upper triangular
  $m\times m$ array ($m\geq 0$) of ballots on pairwise disjoint sets
  with union $U$, such that every row contains at least one element
  of $U$. We write $\cB[U]$ for the set of ballot matrices on $U$.
  Transport along $\sigma\colon U\to V$ replaces every block $B$
  by $\sigma(B)$. A ballot matrix of \emph{dimension} $m$ with $k$ blocks
  in total is \emph{positive} or \emph{negative} according as $k+m$ is even
  or odd. We write $\cB^+[U]$ and $\cB^-[U]$ for the corresponding
  subsets of $\cB[U]$.
\end{definition}

Note that $m\leq|U|$, so $\cB[U]$ is finite. Positivity and negativity
are preserved by transport, so $\cB^+$ and $\cB^-$ are subspecies of
$\cB$. Here, positive refers to the sign of a ballot matrix, not to
positivity in $\Virt$, and $\cB^{+}$ and $\cB^{-}$ are not the two
parts of the reduced form of $\BM$. In fact, we shall prove that
$\cB^{+}=\cB^{-}+\I$; see Corollary~\ref{cor:positivity}.

\begin{proposition}[Claesson and Hannah~\cite{CH2014}]\label{prop:ballot}
  The virtual species $\BM$ defined by \eqref{eq:balmat} satisfies
  \[
    \BM=\cB^+-\cB^-.
  \]
\end{proposition}

\begin{proof}
  The product $(E^{-1})^i$ is represented by the sequences of $i$
  pairwise disjoint ballots, such a sequence with $k$ blocks in
  total having sign $(-1)^k$.
  Subtracting $1$ removes the sequence in which all $i$ ballots are empty.
  Take the product over $i=1,\dots,m$, matching the factor
  $(E^{-1})^i-1$ with the row of length $i$ in an upper triangular
  $m\times m$ matrix. The $m$th term of \eqref{eq:balmat} then
  consists of the ballot matrices of dimension $m$, each with sign
  $(-1)^{m+k}$. Summing over $m$ concludes the proof.
\end{proof}

\section{Interval orders and composition matrices}\label{sec:canonical}

By a \emph{poset} on a finite set $U$ we mean a strict partial order $<$ on
$U$, that is, an irreflexive transitive relation. For $x\in U$ write
\[
  D(x)=\{y:y<x\}
\]
for the \emph{strict down-set} of $x$. The \emph{strict up-set} of $x$
is $\{y:x<y\}$.
By Fishburn's theorem, the finite
$(2{+}2)$-free posets are exactly the finite interval orders;
Bogart~\cite{Bogart1993} gives a short proof.

\begin{definition}
  The species $\I$ of interval orders assigns to a finite set $U$ the
  set $\I[U]$ of $(2{+}2)$-free posets on $U$, here viewed as sets of ordered
  pairs. A bijection $\sigma\colon U\to V$ transports $P\in\I[U]$ to
  \[
    \sigma\!\cdot\!P=\{(\sigma x,\sigma y):(x,y)\in P\}\in\I[V].
  \]
\end{definition}

Since being $(2{+}2)$-free is invariant under relabeling, $\I$ is a
subspecies of the species of posets.

Claesson, Dukes and Kubitzke~\cite{CDK2011} define a \emph{composition
matrix} on a finite set $U$ to be an upper triangular matrix of subsets
of $U$ whose nonempty entries partition $U$ and in which every row and
every column has a nonempty entry. For instance,
\[
  \begin{bmatrix}
    \{1\} & \emptyset & \{2,5\}\\
          & \{3\}     & \emptyset\\
          &           & \{4\}
  \end{bmatrix}
\]
is a $3\times 3$ composition matrix on $[5]$.

An element of $U$ lies in exactly one entry, so an $m\times m$
composition matrix is the same thing as a map recording, for each
element, the position of the entry that contains it. That is the form
we use.

\begin{definition}\label{def:cm}
  For $m\geq 0$ let
  $\Int_m=\{(a,b)\in[m]\times[m]: a\leq b\}$. We regard $(a,b)$ both as
  the position in row $a$ and column $b$ of an upper triangular matrix
  and as the interval $[a,b]\subseteq[m]$, whence the notation. The
  species $\CM{m}$ of $m\times m$ composition matrices assigns to $U$
  the set of maps
  \[
    C\colon U\to\Int_m
  \]
  such that both coordinate maps are surjective onto $[m]$. That is,
  if we write $C(x)=(a(x),b(x))$, then both $x\mapsto a(x)$ and
  $x\mapsto b(x)$ are surjective onto $[m]$. The transport of $C$
  along a bijection $\sigma\colon U\to V$ is $C\circ\sigma^{-1}$.
\end{definition}

The entry in position $(a,b)\in\Int_m$ is the fiber
$C^{-1}(a,b)$. These fibers are pairwise disjoint and have union $U$.
The two surjectivity conditions say that no row and no column is
empty. So the $\CM{m}$-structures on $U$ are the $m\times m$
composition matrices on $U$. Since $\CM{m}[U]=\emptyset$ when $m>|U|$,
only finitely many summands are nonempty on each finite set $U$, and
\[
  \CMsum=\sum_{m\geq 0}\CM{m}
\]
is a well-defined species.

We may also read a composition matrix directly as a family of intervals.
If $C(x)=(a,b)$, then the interval assigned to $x$ is $[a,b]$. Every value
in $[m]$ occurs as a left endpoint and as a right endpoint. Thus,
composition matrices are the same objects as families of intervals
with these two endpoint conditions.

\begin{lemma}[Fishburn~\cite{Fishburn1970}]\label{lem:chain}
  A finite poset $P$ is $(2{+}2)$-free if and only if its strict
  down-sets $\{D(x):x\in P\}$ form a chain under inclusion.
\end{lemma}

\begin{proof}
  Suppose that $D(x)$ and $D(y)$ are incomparable. Choose
  $a\in D(x)\setminus D(y)$ and $c\in D(y)\setminus D(x)$. Then
  $a<x$ and $c<y$, while $a\not<y$ and $c\not<x$. Any equality or
  comparison between the two chains would, by transitivity, give one of
  these two forbidden relations. Thus, the four elements are distinct
  and induce a $2{+}2$.
  Conversely, if $a<b$ and $c<d$ induce a $2{+}2$, then
  $a\in D(b)\setminus D(d)$ and $c\in D(d)\setminus D(b)$. Thus,
  $D(b)$ and $D(d)$ are incomparable.
\end{proof}

Let $P\in \I[U]$. By Lemma~\ref{lem:chain}, its
distinct strict down-sets can be written
\[
  D_1\subset\cdots\subset D_m.
\]
For $x\in U$, let $\ell(x)$ be its \emph{level}, so that $D_{\ell(x)}$
is its strict down-set, and put
\[
  r(x)=\max\{k\in[m]:x\notin D_k\}.
\]
Finally, define $\Gamma(P):U\to \Int_m$ by $\Gamma(P)(x)=\bigl(\ell(x),r(x)\bigr)$.

\begin{lemma}[Dukes, Jel\'{\i}nek and Kubitzke~\cite{DJK2011}]\label{lem:djk}
  If $P\in\I[U]$ has $m$ distinct strict down-sets, then
  $\Gamma(P)\in\CM{m}[U]$, and $x<y$ in $P$ if and only if
  $r(x)<\ell(y)$. Thus, the intervals $[\ell(x),r(x)]$ represent
  $P$. The map $\Gamma\colon\I[U]\to\CMsum[U]$ is a bijection. If
  $C(x)=(a(x),b(x))$, its inverse is
  \[
    \Gamma^{-1}(C)=\bigl\{(x,y): b(x)<a(y)\bigr\}.
  \]
\end{lemma}

\begin{proof}
  The result is clear if $U$ is empty. Suppose that $U$ is nonempty
  and $P\in\I[U]$ has $m$ distinct strict down-sets
  $D_1\subset D_2\subset\cdots\subset D_m$. Here $D_1=\emptyset$, so the
  set defining $r(x)$ is nonempty. Since the sets $D_k$ form a chain,
  \begin{equation}\label{eq:order-criterion}
    x<y
    \;\iff\;x\in D_{\ell(y)}
    \;\iff\;r(x)<\ell(y).
  \end{equation}
  Taking $y=x$ and using irreflexivity gives $\ell(x)\leq r(x)$.

  The map $\ell$ is surjective by construction. If $k<m$, then
  $r(x)=k$ for any $x\in D_{k+1}\setminus D_k$. For $k=m$, choose $x$
  with $\ell(x)=m$; then $m=\ell(x)\leq r(x)\leq m$. Thus, $r$ is also
  surjective, and $\Gamma(P)$ is a composition matrix.

  Conversely, let $C\in\CM{m}[U]$, with $C(x)=(a(x),b(x))$, and define
  $P_C$ by $x<y$ if and only if $b(x)<a(y)$. The intervals
  $[a(x),b(x)]$ represent $P_C$, so it is an interval order. The strict
  down-set of $y$ in $P_C$ is $\{x:b(x)<a(y)\}$.
  Since $b$ is surjective, the sets
  \[
    D'_{\mkern-1mu k}=\{x:b(x)<k\}\qquad(k\in[m])
  \]
  form a chain of $m$ distinct sets. Since $a$ is surjective,
  each $D'_{\mkern-1mu k}$ occurs as a down-set, and the down-set of $y$ is
  $D'_{\!a(y)}$. Hence,
  $\ell(y)=a(y)$. For fixed $x$, the set $D'_{\mkern-1mu k}$ does not contain $x$
  precisely when $k\leq b(x)$. Thus, $r(x)=b(x)$ and
  $\Gamma(P_C)=C$. The criterion~\eqref{eq:order-criterion} also gives
  $P_{\Gamma(P)}=P$, and the two constructions are mutually inverse.
\end{proof}

Equivalently, when the distinct strict up-sets are ordered by reverse
inclusion, $r(x)$ is the level of the strict up-set of $x$.

We call $\Gamma(P)$ the \emph{composition matrix of $P$}.

We read the criterion $r(x)<\ell(y)$ off the matrix as follows: the hook
from $x$ down its column to the row of $y$, and then along that row to
$y$, turns strictly below the diagonal:
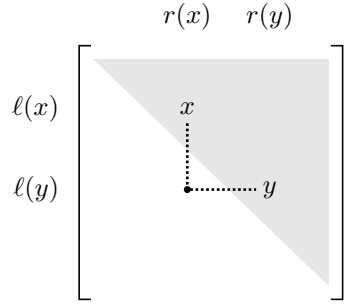
\[
\begin{tikzpicture}[x=5.2mm,y=5.0mm,
    baseline=(current bounding box.center),
    every node/.style={font=\small}]
  \fill[black!10] (0,0) -- (6,0) -- (6,-6) -- cycle;
  \draw[line width=.6pt]
    (-.05,.35) -- (-.35,.35) -- (-.35,-6.35) -- (-.05,-6.35);
  \draw[line width=.6pt]
    (6.05,.35) -- (6.35,.35) -- (6.35,-6.35) -- (6.05,-6.35);
  \draw[hook] (2.4,-1.7) -- (2.4,-3.45) -- (4.15,-3.45);
  \fill (2.4,-3.45) circle (1.3pt);
  \node at (2.4,-1.3) {$x$};
  \node at (4.5,-3.45) {$y$};
  \node[anchor=south] at (2.4,.5) {$r(x)$};
  \node[anchor=south] at (4.5,.5) {$r(y)$};
  \node[anchor=east] at (-.6,-1.3) {$\ell(x)$};
  \node[anchor=east] at (-.6,-3.45) {$\ell(y)$};
\end{tikzpicture}
\]

Claesson, Dukes and Kubitzke~\cite{CDK2011} gave an earlier bijection
between composition matrices and $(2{+}2)$-free posets on $[n]$,
factoring through ascent sequences~\cite{BCDK2010,DP2010}. Dukes,
Jel\'{\i}nek and Kubitzke~\cite{DJK2011} remark that it agrees with
$\Gamma$.

\begin{example}\label{ex:canon}
  Let $P$ on $U=[6]$ be the poset drawn below. The distinct strict
  down-sets are
  $D_1=\emptyset$, $D_2=\{1\}$, $D_3=\{1,2\}$ and $D_4=\{1,2,3,4\}$,
  so $m=4$ and $(\ell(1),\dots,\ell(6))=(1,2,2,2,3,4)$. For
  $x=1,\dots,6$, the largest down-sets not containing $x$ are, respectively,
  $D_1$, $D_2$, $D_3$, $D_3$, $D_4$ and $D_4$. Hence,
  $(r(1),\dots,r(6))=(1,2,3,3,4,4)$. The three equivalent views are shown
  below.
  \[
  \begin{array}{c@{\quad}c@{\quad}c@{\quad}c@{\quad}c}
    \begin{tikzpicture}[
        baseline=(current bounding box.center),
        x=7mm,y=8mm,semithick,every node/.style={font=\small}]
      \node[vtx] (p1) at (.6,0) {};
      \node[vtx] (p2) at (-.6,1) {};
      \node[vtx] (p3) at (.6,1) {};
      \node[vtx] (p4) at (1.8,1) {};
      \node[vtx] (p5) at (-.6,2) {};
      \node[vtx] (p6) at (.6,2) {};
      \draw (p1)--(p2) (p1)--(p3) (p1)--(p4)
            (p2)--(p5) (p2)--(p6) (p3)--(p6) (p4)--(p6);
      \node[below=2pt] at (p1) {$1$};
      \node[left=2pt]  at (p2) {$2$};
      \node[right=2pt] at (p3) {$3$};
      \node[right=2pt] at (p4) {$4$};
      \node[above=2pt] at (p5) {$5$};
      \node[above=2pt] at (p6) {$6$};
    \end{tikzpicture}
    &\quad &
    \begin{tikzpicture}[
        baseline=(current bounding box.center),
        x=9mm,y=3.3mm,every node/.style={font=\small}]
      \foreach \x in {1,2,3,4}{
        \draw[black!30,line width=.4pt] (\x,.55)--(\x,6.5);
        \node[below] at (\x,.5) {$\x$};}
      \foreach \a/\b/\h in {1/1/1, 2/2/2, 2/3/3, 2/3/4, 3/4/5, 4/4/6}{
        \draw[line width=1pt] (\a,\h)--(\b,\h);
        \fill (\a,\h) circle (1.2pt) (\b,\h) circle (1.2pt);
        \node[right=3pt] at (\b,\h) {$\h$};}
    \end{tikzpicture}
    &\quad &
    \begin{bmatrix}
      \{1\} & \emptyset & \emptyset & \emptyset\\
            & \{2\}     & \{3,4\}   & \emptyset\\
            &           & \emptyset & \{5\}\\
            &           &           & \{6\}
    \end{bmatrix}
    \\[45pt]
    \textit{poset} & & \textit{interval model\quad} & & \textit{composition matrix}
  \end{array}
  \]
  The interval model is read directly from the matrix: for instance, $3$
  lies in position $(2,3)$ and is therefore represented by $[2,3]$.
  Lemma~\ref{lem:djk} shows that this is the unique interval model of $P$
  in which every integer in $[4]$ occurs as both a left and a right
  endpoint.
\end{example}

\begin{theorem}\label{thm:canonical}
  The map $\Gamma$ is a natural isomorphism $\I=\CMsum$.
\end{theorem}

\begin{proof}
  By Lemma~\ref{lem:djk} the map $\Gamma$ is a bijection between
  $\I[U]$ and $\CMsum[U]$. It remains to prove naturality. Let
  $\sigma\colon U\to V$ be a bijection. We need to prove that the square
  \[
    \begin{tikzpicture}[x=27mm,y=22mm,
        baseline=(current bounding box.center),
        >={Stealth[length=5pt]}]
      \node (iu) at (0,0)   {$\I[U]$};
      \node (cu) at (1,0)   {$\CMsum[U]$};
      \node (iv) at (0,-1)  {$\I[V]$};
      \node (cv) at (1,-1)  {$\CMsum[V]$};
      \draw[->] (iu) -- node[above] {$\Gamma$} (cu);
      \draw[->] (iv) -- node[below] {$\Gamma$} (cv);
      \draw[->] (iu) -- node[left]  {$\I[\sigma]$} (iv);
      \draw[->] (cu) -- node[right] {$\CMsum[\sigma]$} (cv);
    \end{tikzpicture}
  \]
  commutes. Let $P\in\I[U]$. Strict down-sets satisfy
  $D_{\sigma\cdot P}(\sigma x)=\sigma(D_P(x))$. Thus, transport preserves
  the chain $D_1,\dots,D_m$, its indices, and membership in its sets. It
  therefore preserves both $\ell$ and $r$:
  \[
    \ell_{\sigma\cdot P}(\sigma x)=\ell_P(x),
    \qquad
    r_{\sigma\cdot P}(\sigma x)=r_P(x).
  \]
  Hence, $\Gamma(\sigma\cdot P)=\Gamma(P)\circ\sigma^{-1}$, which is the
  transport of Definition~\ref{def:cm}. Thus, the square commutes and
  $\Gamma$ is a natural isomorphism.
\end{proof}

\section{An algebraic proof of Theorem~\ref{thm:intro}}\label{sec:main}

A composition matrix asks two things of an upper triangular matrix:
that no row be empty and that no column be empty. We shall keep the first
condition and impose the second by inclusion-exclusion over the set of
columns used.
For $m\geq 0$, an $m\times m$ \emph{quasi composition
matrix} on $U$ is an upper triangular matrix of subsets of $U$ whose
nonempty entries partition $U$ and in which every row has a nonempty
entry. It is a composition matrix exactly when every column has a nonempty
entry as well.

For a fixed $m\geq 0$ and $K\subseteq[m]$, let $W_K$ be the species of
$m\times m$ quasi composition matrices that use no column outside $K$.
The matrix being upper triangular, row $a$ has
\[
  \nu_K(a)=\#\{b\in K: b\geq a\}
\]
positions at its disposal.

\begin{lemma}\label{lem:columns}
  Let $m\geq 0$ and let $K\subseteq[m]$. Then
  \[
    W_K\;=\;\prod_{a=1}^{m}\bigl(E^{\nu_K(a)}-1\bigr).
  \]
\end{lemma}

\begin{proof}
  In the form of Definition~\ref{def:cm}, $W_K$ is the species of those
  $C\colon U\to\Int_m$ whose first coordinate map is surjective onto
  $[m]$ and whose second coordinate map takes its values in $K$. Both
  conditions are preserved by transport. The elements in row $a$ form
  a nonempty set, and each may be placed in any of the $\nu_K(a)$
  available positions of that row.
  For a finite linearly ordered set $T$, the species of $T$-valued maps
  is $E^{|T|}$: the natural isomorphism sends a map to its tuple of fibers.
  Thus, the species of $T$-valued maps on a nonempty set is $E^{|T|}-1$.
  Multiplying the contributions from the rows gives the claimed formula.
\end{proof}

\begin{lemma}\label{lem:sieve}
  Let $m\geq 0$. Then, in $\Virt$,
  \[
    \CM{m} \;=\; \sum_{k\geq 0}
    \sum_{\substack{\gamma_1,\dots,\gamma_k\geq 1\\
        \gamma_1+\cdots+\gamma_k=m}}
    (-1)^{m-k}\prod_{i=1}^{k}\bigl(E^{i}-1\bigr)^{\gamma_i}.
  \]
\end{lemma}

\begin{proof}
  Both sides are $1$ when $m=0$, so let $m\geq 1$. Let $W^{=}_{K}$ be
  the subspecies of $W_K$ consisting of those quasi composition matrices
  whose set of nonempty columns is exactly $K$. Then
  $W_K=\sum_{J\subseteq K}W^{=}_{J}$ and $W^{=}_{[m]}=\CM{m}$, so the
  inclusion-exclusion principle gives
  \begin{equation}\label{eq:inex}
    \CM{m}=\sum_{K\subseteq[m]}(-1)^{m-|K|}W_K.
  \end{equation}
  If $m\notin K$, then $\nu_K(m)=0$, so Lemma~\ref{lem:columns} gives
  $W_K=0$. Thus only the subsets containing $m$ contribute.
  Let $K=\{b_1<\cdots<b_k=m\}$ and put $b_0=0$. If
  $b_{j-1}<a\leq b_j$, then $\nu_K(a)=k-j+1$. Thus, on letting
  \[
    \gamma_{k-j+1}=b_j-b_{j-1},
  \]
  we obtain positive integers $\gamma_1,\dots,\gamma_k$ that sum to $m$,
  and
  \[
    W_K=\prod_{i=1}^{k}(E^i-1)^{\gamma_i}.
  \]
  For instance, let $m=7$ and $K=\{2,3,6,7\}$, so that $k=4$. The
  shaded positions below are those in the columns outside $K$, which
  must stay empty. The runs of rows with equally many positions are
  braced on the left and the corresponding factors on the right.
  \[
  \begin{tikzpicture}[x=.52cm,y=.48cm,
      baseline=(current bounding box.center),
      every node/.style={font=\small},
      decoration={brace,amplitude=4pt}]
    \foreach \c in {1,4,5}
      \foreach \r in {1,...,\c} \fill[forbidden] (\c,-\r) rectangle ++(1,-1);
    \foreach \r in {1,...,7}
      \foreach \c in {\r,...,7} \draw[cell] (\c,-\r) rectangle ++(1,-1);
    \foreach \c in {1,...,7}
      \node[anchor=south] at ({\c+.5},-.95) {$\c$};
    \draw[decorate] (.72,-2.94) -- (.72,-1.06);
    \draw[decorate] (.72,-3.94) -- (.72,-3.06);
    \draw[decorate] (.72,-6.94) -- (.72,-4.06);
    \draw[decorate] (.72,-7.94) -- (.72,-7.06);
    \draw[decorate] (8.3,-1.06) -- (8.3,-2.94);
    \draw[decorate] (8.3,-3.06) -- (8.3,-3.94);
    \draw[decorate] (8.3,-4.06) -- (8.3,-6.94);
    \draw[decorate] (8.3,-7.06) -- (8.3,-7.94);
    \node[anchor=east] at (.38,-2)   {$\gamma_4=2$};
    \node[anchor=east] at (.38,-3.5) {$\gamma_3=1$};
    \node[anchor=east] at (.38,-5.5) {$\gamma_2=3$};
    \node[anchor=east] at (.38,-7.5) {$\gamma_1=1$};
    \node[anchor=west] at (8.74,-2)   {$(E^4-1)^{\gamma_4}$};
    \node[anchor=west] at (8.74,-3.5) {$(E^3-1)^{\gamma_3}$};
    \node[anchor=west] at (8.74,-5.5) {$(E^2-1)^{\gamma_2}$};
    \node[anchor=west] at (8.74,-7.5) {$(E^1-1)^{\gamma_1}$};
  \end{tikzpicture}
  \]

  Conversely, a composition $(\gamma_1,\dots,\gamma_k)$ recovers $K$ by
  \[
    b_j=\gamma_k+\gamma_{k-1}+\cdots+\gamma_{k-j+1}.
  \]
  Hence $K\mapsto(\gamma_1,\dots,\gamma_k)$ is a bijection between the
  $k$-element subsets of $[m]$ containing $m$ and the compositions of $m$
  into $k$ parts. Substitution in~\eqref{eq:inex}
  proves the result.
\end{proof}

For instance, the compositions of $m=2$ are $2$ and $1+1$, so that
\[
  \CM{2}=-(E-1)^2+(E-1)(E^2-1)=E\cdot E_+^2,
\]
which agrees with the following direct description. A $2\times 2$
composition matrix has a nonempty entry at $(1,1)$, an arbitrary entry
at $(1,2)$, and a nonempty entry at $(2,2)$.

\begin{proof}[Proof of Theorem~\ref{thm:intro}]
  We have
  \begin{align*}
    \I
    &=\sum_{m\geq 0}\CM{m} && \text{(Theorem~\ref{thm:canonical})} \\
    &=\sum_{m\geq 0}\sum_{k\geq 0}
      \sum_{\substack{\gamma_1,\dots,\gamma_k\geq 1\\
             \gamma_1+\cdots+\gamma_k=m}}
      (-1)^{m-k}\prod_{i=1}^{k}\bigl(E^{i}-1\bigr)^{\gamma_i}
      && \text{(Lemma~\ref{lem:sieve})} \\
    &=\sum_{k\geq 0}\,
      \sum_{\gamma_1,\dots,\gamma_k\geq1}\,
      \prod_{i=1}^{k}
      (-1)^{\gamma_i-1}\bigl(E^i-1\bigr)^{\gamma_i}
      && \text{($\textstyle{m-k=\sum_{i=1}^{k}(\gamma_i-1)}$)} \\
    &=\sum_{k\geq 0}\prod_{i=1}^{k}
      \sum_{\gamma\geq 1}(-1)^{\gamma-1}
      \bigl(E^i-1\bigr)^{\gamma}
      && \text{(distributivity)} \\
    &=\sum_{k\geq 0}\prod_{i=1}^{k}
      \Bigl(1-\sum_{\gamma\geq 0}(-1)^{\gamma}\bigl(E^i-1\bigr)^{\gamma}\Bigr)\\
    &=\sum_{k\geq 0}\prod_{i=1}^{k}
      \Bigl(1-\bigl(1+(E^i-1)\bigr)^{-1}\Bigr) \\
    &=\sum_{k\geq 0}\prod_{i=1}^{k}\bigl(1-(E^{-1})^i\bigr) = \BM && \qedhere
  \end{align*}
\end{proof}

\begin{corollary}\label{cor:positivity}
  The virtual species $\BM$ is positive. In the notation of
  Proposition~\ref{prop:ballot}, there is an isomorphism of species
  \[
    \cB^{+}
    \;=\;
    \cB^{-}+\I.
  \]
\end{corollary}

\begin{proof}
  By Proposition~\ref{prop:ballot} and Theorem~\ref{thm:intro},
  $\cB^+-\cB^-=\I-0$, and
  the defining equivalence relation for $\Virt$ therefore gives
  $\cB^+=\cB^-+\I$.
\end{proof}

Thus, there is a family of bijections
$\cB^{+}[U]\to\cB^{-}[U]\sqcup\I[U]$, one for each finite set $U$, that
commutes with transport. The involution $\eta$ of Claesson and
Hannah~\cite{CH2014} gives a bijection for each linearly ordered $U$, but
that family does not commute with all bijections. For example, $\eta$ sends
the $1\times1$ matrix with entry $\{1,2\}$ to the matrix with entry
$\{1\}\{2\}$. The transposition of $1$ and $2$ fixes the first matrix but
not the second. In the next section we construct a natural bijection.

\section{A bijective proof of Theorem~\ref{thm:intro}}
\label{sec:involution}

We now prove Corollary~\ref{cor:positivity} bijectively. In view of
Proposition~\ref{prop:ballot}, we seek a natural sign-reversing involution
on ballot matrices whose fixed points are positive and naturally in
bijection with interval orders.

Claesson and Hannah~\cite{CH2014} associate an interval
order $P(A)$ with every ballot matrix $A$ on $U$. Let
$(\Row(x),\Col(x))$ be the position of the entry
containing $x\in U$. Then
\[
  x<y \text{ in } P(A)
  \;\iff\;
  \Col(x)<\Row(y).
\]
Since the matrix is upper triangular,
$\Row(x)\leq\Col(x)$. Let
$I_x=[\Row(x),\Col(x)]$.
Since $\max I_x=\Col(x)$ and
$\min I_y=\Row(y)$, the displayed equivalence says precisely
that $x<y$ exactly when $I_x$ lies entirely to the left of $I_y$. Thus, the
intervals $I_x$ represent $P(A)$.

Making each nonempty entry of a composition matrix $C$ into a one-block
ballot gives a ballot matrix $A_C$. The formula for $\Gamma^{-1}$
in Lemma~\ref{lem:djk} then gives $P(A_C)=\Gamma^{-1}(C)$.

\subsection{The involution}
Fix a ballot matrix $A$ of dimension $m$. For each row $i$, let $f_i$ be
the column of its first nonempty entry, and let $S\subseteq[m]$ be the set
of empty columns. Every row is nonempty, so $f_i$ is well defined. If
$m\geq1$, the only position in row $m$ is $(m,m)$, and that entry is
nonempty. Thus, $m\notin S$. Define:
\begin{align*}
  \text{row $i$ is \emph{splittable}}
  &\iff \text{row $i$ contains at least two blocks;} \\
  \text{row $i$ is \emph{mergeable}}
  &\iff
    \text{row $i$ contains one block, $i\in S$ and $f_{i+1}\geq f_i$.}
\end{align*}
Here $i\in S$ forces $i<m$. No row is both splittable and mergeable.

For a \emph{split at row $i$}, let $(i,j)$ be the first nonempty entry and $B$
the first block there. Remove $B$, insert an empty row and column at
index $i$, and place the one-block ballot $B$ at $(i,j+1)$. The remainder
of that entry moves to $(i+1,j+1)$. For a \emph{merge at row $i$}, let $B$
be its block. Since column $i$ is empty, $f_i>i$. Move $B$ to
$(i+1,f_i)$, ahead of any blocks already there, and delete row and column
$i$. See Figure~\ref{fig:splitmerge}.

\begin{figure}[ht]
  \centering
  \begin{tikzpicture}[
      x=.66cm,y=.52cm,
      every node/.style={font=\footnotesize},
      >={Stealth[length=5pt]}
    ]
    \begin{scope}
      \fill[accent] (4,-2) rectangle ++(1,-1);
      \foreach \r in {1,...,5}
        \foreach \c in {\r,...,5}
          \draw[cell] (\c,-\r) rectangle ++(1,-1);
      \node at (4.5,-2.5) {$B\beta$};
      \node at (5.5,-2.5) {$\gamma$};
      \node[anchor=east] at (1.85,-2.5) {$i$};
    \end{scope}

    \draw[->] (7.2,-2.15) -- node[above,font=\scriptsize] {split}
      (10.2,-2.15);
    \draw[<-] (7.2,-2.85) -- node[below,font=\scriptsize] {merge}
      (10.2,-2.85);

    \begin{scope}[shift={(10.4,0)}]
      \foreach \c in {2,...,6} \fill[forbidden] (\c,-2) rectangle ++(1,-1);
      \foreach \r in {1,2}     \fill[forbidden] (2,-\r) rectangle ++(1,-1);
      \fill[accent] (5,-2) rectangle ++(1,-1);
      \fill[forbidden] (3,-3) rectangle ++(2,-1);
      \foreach \r in {1,...,6}
        \foreach \c in {\r,...,6}
          \draw[cell] (\c,-\r) rectangle ++(1,-1);
      \node at (5.5,-2.5) {$B$};
      \node at (5.5,-3.5) {$\beta$};
      \node at (6.5,-3.5) {$\gamma$};
      \node[anchor=east] at (1.85,-2.5) {$i$};
      \node[anchor=south] at (2.5,-.9) {$i$};
    \end{scope}
  \end{tikzpicture}
  \caption{A split and its inverse merge. Here $\beta$ is the remainder
    of the source ballot and $\gamma$ a possible later ballot in the same
    row. Splittability ensures that at least one is present. The shaded
    positions are empty.}
  \label{fig:splitmerge}
\end{figure}
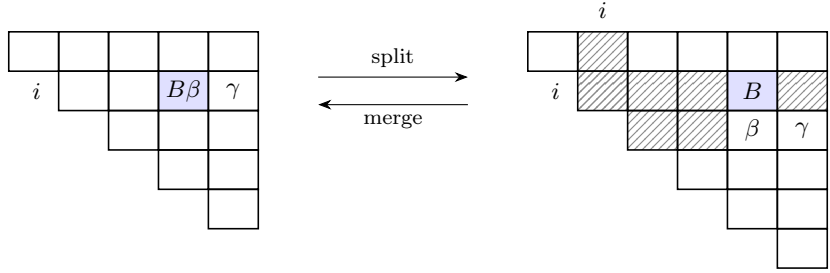
\FloatBarrier

\begin{example}\label{ex:moves}
  Let $U=[8]$. In $A$, row $2$ is the first row admitting a move.
  Splitting off the block $\{2,5\}$ gives $A'$, and merging row $2$
  returns $A$. The gray positions mark the inserted row and column.
  \[
  \setlength{\arraycolsep}{3pt}
  A=\left[
    \begin{array}{cccc}
      \{1\} & \emptyset & \emptyset & \emptyset\\
      & \emptyset & {\color{blue}\{2,5\}}\{7\} & \emptyset\\
      & & \emptyset & \{3,8\}\\
      & & & \{4\}\{6\}
    \end{array}\right]
  \quad
  \begin{array}{c}
    \xrightarrow{\ \text{split}\ }\\[-1mm]
    \xleftarrow[\ \text{merge}\ ]{}
  \end{array}
  \quad
  \left[
    \begin{array}{ccccc}
      \{1\}&{\color{gray}\emptyset}&\emptyset&\emptyset&\emptyset\\
      &{\color{gray}\emptyset}&{\color{gray}\emptyset}&
        {\color{blue}\{2,5\}}&{\color{gray}\emptyset}\\
      &&\emptyset&\{7\}&\emptyset\\
      &&&\emptyset&\{3,8\}\\
      &&&&\{4\}\{6\}
    \end{array}\right]=A'.
  \]
  There are six blocks in both matrices, so their signs are positive
  and negative, respectively. Both matrices determine the interval order
  
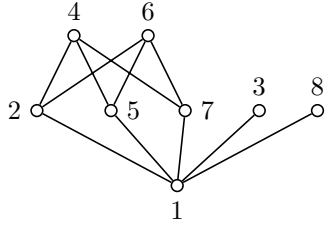
\[
  \begin{tikzpicture}[
      baseline=(current bounding box.center),
      x=7mm,y=9mm,semithick,every node/.style={font=\small}]
    \node[vtx] (q1) at (2.65,0)  {};
    \node[vtx] (q2) at (0,1.1)   {};
    \node[vtx] (q5) at (1.4,1.1) {};
    \node[vtx] (q7) at (2.8,1.1) {};
    \node[vtx] (q3) at (4.2,1.1) {};
    \node[vtx] (q8) at (5.3,1.1) {};
    \node[vtx] (q4) at (.7,2.2)  {};
    \node[vtx] (q6) at (2.1,2.2) {};
    \draw (q1)--(q2) (q1)--(q5) (q1)--(q7) (q1)--(q3) (q1)--(q8)
          (q2)--(q4) (q2)--(q6) (q5)--(q4) (q5)--(q6)
          (q7)--(q4) (q7)--(q6);
    \node[below=2pt] at (q1) {$1$};
    \node[left=2pt]  at (q2) {$2$};
    \node[right=2pt] at (q5) {$5$};
    \node[right=2pt] at (q7) {$7$};
    \node[above=2pt] at (q3) {$3$};
    \node[above=2pt] at (q8) {$8$};
    \node[above=2pt] at (q4) {$4$};
    \node[above=2pt] at (q6) {$6$};
  \end{tikzpicture}
  \]
  We return to this in Example~\ref{ex:fixedpoint} below.
\end{example}

\begin{definition}
  For $A\in\cB[U]$, let $\theta(A)$ be the result of performing the move
  (split or merge) at the least row admitting one. If no move is
  available, let $\theta(A)=A$.
\end{definition}

The move conditions, the chosen row and the moves themselves use only
occupied positions and block order, not labels. Thus,
$\theta(\sigma\!\cdot\!A)=\sigma\!\cdot\!\theta(A)$ for every bijection
$\sigma$, and $\Fix(\theta)$ is a subspecies of $\cB$.

\begin{theorem}\label{thm:involution}
  For every ballot matrix $A$, we have $\theta(\theta(A))=A$ and
  $P(\theta(A))=P(A)$. If $\theta(A)\neq A$, then $A$ and $\theta(A)$
  have opposite signs.
\end{theorem}

\begin{proof}
  Let $A$ be a ballot matrix of dimension $m$.
  Consider a split at row $i$. Let $A'$ be the result, let $B$ be the
  block moved by the split, and let
  $\iota\colon[m]\to[m+1]\setminus\{i\}$ be the increasing bijection,
  given by $\iota(t)=t$ for $t<i$ and $\iota(t)=t+1$ for $t\geq i$.
  Assume $\Row$, $\Col$ and $f_1,\dots,f_m$ are defined as before.
  Write $\Row'$, $\Col'$ and $f'_1,\dots,f'_{m+1}$ for the corresponding
  data of $A'$. Then
  \begin{equation}\label{eq:split-coordinates}
    \Col'(x)=\iota(\Col(x)),
    \quad
    \Row'(x)=
    \begin{cases}
      \iota(\Row(x)),&x\notin B\\
      i,&x\in B
    \end{cases}
  \end{equation}
  and, for $1\leq h\leq i$,
  \begin{equation}\label{eq:split-first}
    f'_h=\iota(f_h).
  \end{equation}

  First, row $i$ of $A'$ holds $B$ alone and column $i$ is empty. Row
  $i+1$, the old row $i$ less $B$, is nonempty, and its entries lie in
  columns $\iota(c)$ with $c\geq f_i$. Hence, $A'$ is a ballot matrix
  and $f'_{i+1}\geq f'_i$. Thus, row $i$ is mergeable
  and merging restores $A$. Conversely, consider
  a merge of $B$ at $(i,j)$. The merge condition gives
  $f_{i+1}\geq f_i=j$, so the old row $i+1$ has no entry before column
  $j$. The inequality $f_i>i$ makes the merged matrix upper triangular.
  Its merged row contains $B$, and every other row is nonempty because it
  comes from a row of $A$. Hence, the result is a ballot matrix. Its row
  $i$ contains at least two blocks, and its first nonempty entry begins
  with $B$. Thus, row $i$ is splittable, and splitting it restores $A$.

  Second, a split preserves the interval order. Let $x,y\in U$. If
  $y\notin B$, then \eqref{eq:split-coordinates} and the fact that
  $\iota$ is increasing give
  \[
    \Col'(x)<\Row'(y)
    \;\iff\;
    \iota(\Col(x))<\iota(\Row(y))
    \;\iff\;
    \Col(x)<\Row(y).
  \]
  If $y\in B$, then $\Row'(y)=\Row(y)=i$, and
  \[
    \Col'(x)<i
    \;\iff\;
    \iota(\Col(x))<i
    \;\iff\;
    \Col(x)<i.
  \]
  Hence $P(A')=P(A)$. Merges preserve the order as well, since they are
  inverse to splits.

  Third, we track the least row admitting a move. For a split at row
  $i$, the number of blocks in row $h<i$ and whether column $h$ is empty
  do not change. Moreover, $h+1\leq i$, so \eqref{eq:split-first} gives
  \[
    f'_{h+1}\geq f'_h
    \;\iff\;
    \iota(f_{h+1})\geq\iota(f_h)
    \;\iff\;
    f_{h+1}\geq f_h.
  \]
  The move conditions therefore show that row $h$ admits the same move,
  if any, before and after the split. Since every merge is inverse to a
  split, the same conclusion holds for a merge.

  If $\theta$ acts first at row $i$, the inverse move is therefore
  available at row $i$ of $\theta(A)$, and no earlier row admits a move.
  Since at most one move is available at a row, $\theta(\theta(A))=A$.

  Finally, each move preserves the number of blocks and changes the
  dimension by one. Thus, each move reverses the sign.
\end{proof}

\subsection{The fixed points}

We shall identify the fixed points of $\theta$ with composition
matrices. First we characterize the ballot matrices fixed by $\theta$.

\begin{proposition}\label{prop:fixed}
  A ballot matrix $A$ of dimension $m$ is fixed by $\theta$ if and
  only if every row contains exactly one block and the block positions
  $(i,f_i)$ satisfy
  \begin{equation}\label{eq:fixedpoint-condition}
    \text{column $i$ empty}\quad\Longrightarrow\quad f_{i+1}<f_i
    \quad(1\leq i<m).
  \end{equation}
  In particular, every fixed point has $m$ blocks and sign
  $(-1)^{2m}=+1$.
\end{proposition}

\begin{proof}
  Since every row is nonempty, a row is not splittable precisely when
  it contains one block. Row
  $i<m$ is mergeable precisely when column $i$ is empty and
  $f_{i+1}\geq f_i$. Condition
  \eqref{eq:fixedpoint-condition} excludes exactly these merges.
\end{proof}

We next define a map $\kappa$ from ballot matrices to composition matrices.
The empty ballot matrix is sent to the empty composition matrix. Let $A$
be a nonempty ballot matrix on $U$ of dimension $m$, let $J\subseteq[m]$
be its set of nonempty columns, and put $d=|J|$. For $i\in[m]$, let
\[
  \pi(i)=1+\#\{j\in J:j<i\}
\]
and define $\kappa(A)\colon U\to\Int_d$ by
\[
  \kappa(A)(x) = \bigl(\pi(\Row_A(x)),\pi(\Col_A(x))\bigr).
\]

Thus, $\kappa$ may send elements from distinct blocks to the same
position. For instance,
\[
  \left[
    \begin{array}{cc}
      \emptyset & \{x\}\\
                & \{y\}
    \end{array}
  \right]
  \xrightarrow{\ \kappa\ }
  \left[\{x,y\}\right].
\]
Here $J=\{2\}$ and $\pi(1)=\pi(2)=1$.
Elements from distinct blocks are not sent to the same position when
$A$ is fixed by $\theta$; see Proposition~\ref{prop:fixcm}.

For a composition matrix $C$, we write $P(C)$ for $P(A_C)$.

\begin{lemma}\label{lem:kappa}
  For every ballot matrix $A$, $\kappa(A)$ is a composition matrix and
  \[
    P(\kappa(A))=P(A).
  \]
  Hence, $\kappa(A)=\Gamma(P(A))$. Moreover, the map $\kappa$ commutes with
  transport.
\end{lemma}

\begin{proof}
  We first show that $\kappa$ maps ballot matrices to composition
  matrices and that the following triangle commutes:
  \[
    \begin{tikzpicture}[x=16mm,y=15mm,
        baseline=(current bounding box.center),
        >={Stealth[length=5pt]}]
      \node (b) at (0,0) {$\cB$};
      \node (c) at (2,0) {$\CMsum$};
      \node (i) at (1,-1) {$\I$};
      \draw[->] (b) -- node[above] {$\kappa$} (c);
      \draw[->] (b) -- node[below left] {$P$} (i);
      \draw[->] (i) -- node[below right] {$\Gamma$} (c);
    \end{tikzpicture}
  \]
  Let $A$ be a ballot matrix on $U$, and let $J$ be its set of nonempty
  columns.
  The result is clear when $A$ is empty. Suppose that $A$ has dimension
  $m\geq 1$, and put $d=|J|$. Its last column is nonempty, so $m\in J$. Thus,
  $\pi$ maps $[m]$ onto $[d]$ and restricts to a bijection $J\to[d]$.
  It follows that $\kappa(A)$ has no empty row or column. Since $i\leq j$
  implies $\pi(i)\leq\pi(j)$, it is upper
  triangular. Its nonempty entries partition the underlying set, so it
  is a composition matrix.

  Let $x$ lie in position $(i,j)$ and $y$ in position $(i',j')$ of $A$.
  Since $j\in J$, the definition of $\pi$ gives
  \[
    \Col_A(x)<\Row_A(y)
    \;\iff\;j<i'
    \;\iff\;\pi(j)<\pi(i').
  \]
  By the definition of $\kappa$, the first and last inequalities are the
  conditions for $x<y$ in $P(A)$ and $P(\kappa(A))$, respectively. Hence,
  $P(\kappa(A))=P(A)$. Lemma~\ref{lem:djk} now gives
  $\kappa(A)=\Gamma(P(A))$, so the triangle commutes.

  Let $\sigma\colon U\to V$ be a bijection. We need to show that
  the following square commutes:
  \[
    \begin{tikzpicture}[x=26mm,y=22mm,
        baseline=(current bounding box.center),
        >={Stealth[length=5pt]}]
      \node (bu) at (0,0)  {$\cB[U]$};
      \node (cu) at (1,0)  {$\CMsum[U]$};
      \node (bv) at (0,-1) {$\cB[V]$};
      \node (cv) at (1,-1) {$\CMsum[V]$};
      \draw[->] (bu) -- node[above] {$\kappa$} (cu);
      \draw[->] (bv) -- node[below] {$\kappa$} (cv);
      \draw[->] (bu) -- node[left] {$\cB[\sigma]$} (bv);
      \draw[->] (cu) -- node[right] {$\CMsum[\sigma]$} (cv);
    \end{tikzpicture}
  \]
  Along the top and right sides, $A$ is sent to
  \[
    \CMsum[\sigma](\kappa(A))=\kappa(A)\circ\sigma^{-1}.
  \]
  Along the left and bottom sides, the ballot
  matrix $\cB[\sigma](A)=\sigma\cdot A$ has the same occupied positions as
  $A$, with every label $x$ replaced by $\sigma x$. Hence, $J$ and $\pi$ are
  unchanged, and, for $x\in U$,
  \[
    (\kappa\circ\cB[\sigma])(A)(\sigma x)
      =\kappa(\sigma\cdot A)(\sigma x)
      =\bigl(\pi(\Row_A(x)),\pi(\Col_A(x))\bigr) =\kappa(A)(x).
  \]
  Thus, the second path also sends $A$ to $\kappa(A)\circ\sigma^{-1}$, and
  the square commutes.
\end{proof}

We next construct a map $\lambda$ that will reverse $\kappa$. A fixed point
has one block in each row, so each nonempty entry of a composition matrix
must acquire a row of its own. We choose its column so that $\kappa$ returns
it to its original position.

Let $\lambda$ send the empty matrix to the empty ballot matrix. Suppose that
$C$ has dimension $d\geq1$ and $m$ nonempty entries. To satisfy
condition~\eqref{eq:fixedpoint-condition}, list them as $C_1,\dots,C_m$ by
increasing row and, within each row, decreasing column. Write
$(a(i),b(i))$ for the position of $C_i$, and let $e_r$ be the number of
entries in rows $1,\dots,r$. Then
\[
  0=e_0<e_1<\cdots<e_d=m,
\]
and the entries in row $r$ are $C_{e_{r-1}+1},\dots,C_{e_r}$.

The entry $C_i$ will occupy row $i$ of $\lambda(C)$. The entries from row
$r$ of $C$ have indices $e_{r-1}+1,\dots,e_r$. To make $\pi$ send these
indices back to $r$, we choose $e_r$ as the $r$-th nonempty column. Thus,
to recover the original column $b(i)$, place the one-block ballot $C_i$ in
position
\begin{equation}\label{eq:lambda-position}
  \bigl(i,e_{b(i)}\bigr).
\end{equation}
For each $r\in[d]$, column $r$ of $C$ contains some entry $C_i$, and then
$b(i)=r$. Thus, the nonempty columns of $\lambda(C)$ are exactly
$e_1,\dots,e_d$, as required.

Since $C_i$ lies in row $a(i)$ of $C$, its index in the list satisfies
$i\leq e_{a(i)}$. Since $C$ is upper triangular, $a(i)\leq b(i)$, and hence
\[
  i\leq e_{a(i)}\leq e_{b(i)}.
\]
Thus, the chosen position lies in the upper triangle. The entries $C_i$
partition the underlying set, and one entry is placed in each row, so
$\lambda(C)$ is a ballot matrix.

\begin{example}\label{ex:fixedpoint}
  Let $P$ be the interval order in Example~\ref{ex:moves}. A fixed point
  $A$ representing $P$ and its image under $\kappa$ are
  \[
  \setlength{\arraycolsep}{3pt}
  A=\left[
    \begin{array}{cccc}
      \{1\} & \emptyset & \emptyset & \emptyset\\
            & \emptyset & \emptyset & \{3,8\}\\
            &           & \{2,5,7\} & \emptyset\\
            &           &           & \{4,6\}
    \end{array}\right]
  \;\xrightarrow{\ \kappa\ }\;
  \left[
    \begin{array}{ccc}
      \{1\} & \emptyset & \emptyset\\
            & \{2,5,7\} & \{3,8\}\\
            &           & \{4,6\}
    \end{array}\right]=C.
  \]
  Here $J=\{1,3,4\}$ and $\pi=(1,2,2,3)$. The four blocks move under
  $\kappa$ to the positions
  \[
    (1,1),\quad(2,3),\quad(2,2),\quad(3,3).
  \]
  In the order used to define $\lambda$, the entries of $C$ are
  \[
    \{1\},\quad\{3,8\},\quad\{2,5,7\},\quad\{4,6\}.
  \]
  Its rows contain $1,2,1$ entries, so $(e_1,e_2,e_3)=(1,3,4)$.
  Formula~\eqref{eq:lambda-position} places these entries in columns
  $1,4,3,4$, respectively, and recovers $A$.
\end{example}

\begin{proposition}\label{prop:fixcm}
  The maps
  \[
    \kappa\colon\Fix(\theta)\to\CMsum
    \qquad\text{and}\qquad
    \lambda\colon\CMsum\to\Fix(\theta)
  \]
  are mutually inverse natural isomorphisms.
\end{proposition}

\begin{proof}
  The result is clear for the empty matrices. Let $C$ be a nonempty
  composition matrix, with the notation above. By construction,
  $\lambda(C)$ has one block in each row and nonempty columns
  $e_1,\dots,e_d$. If column $i$ is empty, then $i<m$, and $C_i$ and
  $C_{i+1}$ lie in the same row of $C$. Hence,
  \[
    f_{i+1}=e_{b(i+1)}<e_{b(i)}=f_i,
  \]
  since the entries in that row were listed by decreasing column.
  Proposition~\ref{prop:fixed} shows that $\lambda(C)$ is fixed.

  We next compute $\kappa(\lambda(C))$. Let $\pi$ be the map associated
  with $\lambda(C)$. Since $C_i$ lies in row $a(i)$ of $C$,
  \[
    e_{a(i)-1}<i\leq e_{a(i)}.
  \]
  The nonempty columns are $e_1<\cdots<e_d$, so
  \[
    \pi(i)=a(i),\qquad \pi(e_{b(i)})=b(i).
  \]
  The block $C_i$ lies at $(i,e_{b(i)})$ in $\lambda(C)$. By the definition
  of $\kappa$, this block is sent to
  \[
    \bigl(\pi(i),\pi(e_{b(i)})\bigr)=(a(i),b(i)),
  \]
  its original position in $C$. Thus, $\kappa(\lambda(C))=C$.

  It remains to show that $\lambda(\kappa(A))=A$. Let $A$ be a nonempty
  fixed point of dimension $m$. Write $B_i$ for the block in row $i$ and
  $(i,f_i)$ for its position. List the nonempty columns as
  $J=\{t_1<\cdots<t_d\}$, and put $t_0=0$. We have $t_d=m$, and
  $\pi(i)=r$ precisely when
  \[
    t_{r-1}<i\leq t_r.
  \]
  If $t_{r-1}<i<i'\leq t_r$, then columns $i,\dots,i'-1$ are empty.
  Condition~\eqref{eq:fixedpoint-condition} therefore gives
  \[
    f_i>f_{i+1}>\cdots>f_{i'}.
  \]
  Since $\pi$ is increasing on $J$, these blocks become distinct entries
  of $\kappa(A)$ in decreasing column order. Thus, the entries of
  $\kappa(A)$, ordered by increasing row and decreasing column, are
  $B_1,\dots,B_m$. Its first $r$ rows contain $t_r$ entries.

  Put $C=\kappa(A)$. In the notation used to define $\lambda$, we have
  $e_r=t_r$. Moreover, $f_i\in J$, so $t_{\pi(f_i)}=f_i$. Therefore,
  $\lambda(C)$ places $B_i$ at
  \[
    \bigl(i,e_{\pi(f_i)}\bigr)=(i,f_i),
  \]
  and $\lambda(\kappa(A))=A$. Finally, $\kappa$ is natural by
  Lemma~\ref{lem:kappa}, while the construction of $\lambda$ uses only
  positions and transports each entry as a whole. Hence, both maps are
  natural.
\end{proof}

\begin{theorem}\label{thm:bijective}
  The involution $\theta$ restricts to a natural isomorphism from the
  non-fixed points of $\cB^{+}$ to $\cB^{-}$, and $A\mapsto P(A)$ is a
  natural isomorphism
  $\Fix(\theta)\to\I$. Consequently,
  \[
    \cB^{+}\;=\;\cB^{-}+\I,
  \]
  and Theorem~\ref{thm:intro} follows again, this time bijectively.
\end{theorem}

\begin{proof}
  By Theorem~\ref{thm:involution}, $\theta$ pairs non-fixed matrices with
  matrices of opposite sign. Proposition~\ref{prop:fixed} shows that the
  fixed points are positive. Lemma~\ref{lem:kappa} and
  Proposition~\ref{prop:fixcm} show that
  $P=\Gamma^{-1}\circ\kappa$ is a natural isomorphism from the fixed points
  to the interval orders. The two isomorphisms give
  $\cB^{+}=\cB^{-}+\I$. Proposition~\ref{prop:ballot} now gives
  $\BM=\I$.
\end{proof}

Let $w_m$ count upper triangular $0$--$1$ matrices of all dimensions
with $m$ ones and no zero row or column. These numbers form the sequence
A138265 in the OEIS~\cite{OEIS}, which begins
$1,1,1,2,5,16,61,271,1372,\dots$

\begin{corollary}\label{cor:fixedseries}
  The fixed-point species is
  \begin{equation}\label{eq:fixballot}
    \Fix(\theta)=\sum_{m\geq 0}w_m\,(E_+)^m.
  \end{equation}
\end{corollary}

\begin{proof}
  Proposition~\ref{prop:fixcm} identifies fixed points with composition
  matrices. Replacing each nonempty entry by $1$ gives one of the $w_m$
  matrices above, and its $m$ ones independently carry nonempty sets.
\end{proof}

\section{Generating series}\label{sec:series}

Let $q=x_1+x_2/2+x_3/3+\cdots$ as before. Theorem~\ref{thm:intro}, together with
$E^{-1}(x)=e^{-x}$,
$\widetilde{E^{-1}}(x)=1-x$, and $Z_{E^{-1}}=e^{-q}$, gives the following
three series.

\begin{corollary}\label{cor:series}\leavevmode
  \begin{enumerate}[label=\textup{(\alph*)}]
    \item The generating series for labeled interval orders is
      \[
        \I(x)=
        \sum_{m\geq 0}\prod_{i=1}^{m}\bigl(1-e^{-ix}\bigr).
      \]
    \item The type generating series is
      \[
        \widetilde{\I}(x)=
        \sum_{m\geq 0}\prod_{i=1}^{m}\bigl(1-(1-x)^i\bigr).
      \]
    \item The cycle index series is
      \[
        Z_{\I}=
        \sum_{m\geq 0}\prod_{i=1}^{m}\bigl(1-e^{-iq}\bigr)
        =\I(q).
      \]
  \end{enumerate}
\end{corollary}

The identities $\I(x)=Z_{\I}(x,0,0,\dots)$ and
$\widetilde{\I}(x)=Z_{\I}(x,x^2,x^3,\dots)$ recover parts
\textup{(a)} and \textup{(b)} from part \textup{(c)}. Indeed, $q$
becomes $x$ in the first and $-\log(1-x)$ in the second.

\section{Automorphisms and the molecular decomposition}
\label{sec:molecular}

We first read the automorphism group from the composition matrix.

\begin{corollary}\label{cor:aut}
  Let $P$ be an interval order on $U$ whose composition matrix has
  dimension $m$, and let $C_{ij}=\{x\in U: (\ell(x),r(x))=(i,j)\}$ be
  its entries. A bijection $\sigma\colon U\to U$ is an automorphism of
  $P$ if and only if $\sigma(C_{ij})=C_{ij}$ for every entry. Thus,
  $\Aut(P)$ is the Young subgroup
  $\prod_{(i,j)\in\Int_m}\SymOn(C_{ij})$, where $\SymOn(C)$ is the
  symmetric group on the set $C$.
\end{corollary}

\begin{proof}
  By naturality and the injectivity of $\Gamma$,
  \[
    \sigma\cdot P=P
    \;\iff\;
    \Gamma(P)\circ\sigma^{-1}=\Gamma(P).
  \]
  The last equality holds if and only if $\sigma$ preserves every fiber
  $C_{ij}$. Such a bijection is obtained by choosing a permutation within
  each entry independently, which gives the stated Young subgroup.
\end{proof}

In Example~\ref{ex:canon}, the only nonsingleton entry is $\{3,4\}$.
Thus $\Aut(P)=\SymOn(\{3,4\})$.

The conclusion is special to interval orders. The automorphism group of
$2{+}2$ has order two but is generated by the double transposition that
swaps the two chains, whereas a Young subgroup of order two is generated
by a transposition.

\begin{corollary}
  A permutation $\sigma\in\Sym_n$ with $c$ cycles is an
  automorphism of exactly $|\I[c]|$ interval orders on $[n]$.
\end{corollary}

\begin{proof}
  By Corollary~\ref{cor:aut}, $\sigma$ fixes $P$ precisely when every
  entry of $\Gamma(P)$ is a union of $\sigma$-cycles. Let $\Omega$ be the
  set of these cycles. Replace each entry $C_{ij}$ by
  \[
    \{B\in\Omega:B\subseteq C_{ij}\}.
  \]
  This gives a composition matrix on $\Omega$. Conversely, replace each
  set of cycles by their union. These constructions are inverse bijections
  between the composition matrices on $[n]$ fixed by $\sigma$ and those on
  $\Omega$. The claim now follows from Theorem~\ref{thm:canonical}.
\end{proof}

For example, let $\sigma\in\Sym_6$ have three cycles. Its cycle type may
be $(2,2,2)$, $(3,2,1)$ or $(4,1,1)$. In every case, exactly
$|\I[3]|=19$ of the $81663$ interval orders in $\I[6]$ are fixed by
$\sigma$.

For an interval order $P$ whose composition matrix has dimension $m$ and
entries $C_{ij}$, define the integer matrix
\[
  N(P)=\Bigl(|C_{ij}|\Bigr)_{(i,j)\in\Int_m}.
\]
Jel\'{\i}nek~\cite{Jelinek2011} calls these
\emph{Fishburn matrices}. They index the
isomorphism types of interval orders and hence the molecules of $\I$.

\begin{corollary}\label{cor:molecular}
  The molecular decomposition of $\I$ is
  \[
    \I\;=\;\sum_{m\geq 0}\sum_N
    \prod_{(i,j)\in\Int_m}E_{N_{ij}},
  \]
  where the inner sum ranges over the upper triangular $m\times m$
  matrices $N$ of nonnegative integers with no zero row and no zero
  column, and factors $E_0=1$ are discarded.
\end{corollary}

\begin{proof}
  Transport preserves entry sizes, so the interval orders with
  Fishburn matrix $N$ form a subspecies $\I_N$. By
  Theorem~\ref{thm:canonical}, an $\I_N$-structure on $U$ is a family
  $(S_{ij})$ of disjoint sets with union $U$ and $|S_{ij}|=N_{ij}$. Hence
  \[
    \I_N=\prod_{(i,j)\in\Int_m}E_{N_{ij}}.
  \]
  Any two such families are isomorphic by matching corresponding entries,
  so $\I_N$ is molecular. The possible matrices $N$ are precisely those
  in the statement, and summing over them proves the result.
\end{proof}

To read off the terms of degree $n$, list the matrices $N$ whose
entries sum to $n$ and replace each by the product of the species
$E_{N_{ij}}$ over its nonzero entries. In degree $3$ the matrices are
\[
  \begin{bmatrix}3\end{bmatrix},\quad
  \begin{bmatrix}2&0\\&1\end{bmatrix},\quad
  \begin{bmatrix}1&0\\&2\end{bmatrix},\quad
  \begin{bmatrix}1&1\\&1\end{bmatrix},\quad
  \begin{bmatrix}1&0&0\\&1&0\\&&1\end{bmatrix},
\]
with products $E_3$, $XE_2$, $XE_2$, $X^3$ and $X^3$, since $E_1=X$.
After we collect equal products and group by degree, the decomposition
begins
\begin{align*}
  \I
  &=1+X+\bigl(E_2+X^2\bigr)
    +\bigl(E_3+2XE_2+2X^3\bigr)\\
  &\qquad\qquad\qquad
    +\bigl(E_4+2XE_3+E_2^2+6X^2E_2+5X^4\bigr)+\cdots
\end{align*}
The coefficient of each molecular species counts the corresponding
isomorphism types. Figure~\ref{fig:molecular} draws them.

\begin{figure}[ht]
  \centering
  \newcommand{\mterm}[1]{\raisebox{1pt}{$#1$}}%
  \renewcommand{\arraystretch}{1.3}
  \begin{tabular}{r@{\quad}|@{\quad}l}
    \mterm{X} & \hassea{0/0/0}\\ \hline
    \mterm{E_2} & \hassea{0/0/0,1/1/0}\\ \hline
    \mterm{X^2} & \hasse{0/0/0,1/0/1}{0/1}\\ \hline
    \mterm{E_3} & \hassea{0/0/0,1/1/0,2/2/0}\\ \hline
    \mterm{XE_2} & \hasse{0/.5/0,1/0/1,2/1/1}{0/1,0/2}%
             \hasse{0/0/0,1/1/0,2/.5/1}{0/2,1/2}\\ \hline
    \mterm{X^3} & \hasse{0/0/0,1/1.2/0,2/0/1}{0/2}%
            \hasse{0/0/0,1/0/1,2/0/2}{0/1,1/2}\\ \hline
    \mterm{E_4} & \hassea{0/0/0,1/1/0,2/2/0,3/3/0}\\ \hline
    \mterm{XE_3} & \hasse{0/1/0,1/0/1,2/1/1,3/2/1}{0/1,0/2,0/3}%
             \hasse{0/0/0,1/1/0,2/2/0,3/1/1}{0/3,1/3,2/3}\\ \hline
    \mterm{E_2^2} & \hasse{0/0/0,1/1/0,2/0/1,3/1/1}{0/2,0/3,1/2,1/3}\\ \hline
    \mterm{X^2E_2} & \hasse{0/.5/0,1/2/0,2/0/1,3/1/1}{0/2,0/3}%
               \hasse{0/0/0,1/1.2/0,2/2.2/0,3/0/1}{0/3}%
               \hasse{0/0/0,1/1/0,2/2.2/0,3/.5/1}{0/3,1/3}%
               \hasse{0/.5/0,1/.5/1,2/0/2,3/1/2}{0/1,1/2,1/3}%
               \hasse{0/.5/0,1/0/1,2/1/1,3/.5/2}{0/1,0/2,1/3,2/3}%
               \hasse{0/0/0,1/1/0,2/.5/1,3/.5/2}{0/2,1/2,2/3}\\ \hline
    \mterm{X^4} & \hasse{0/.5/0,1/0/1,2/1/1,3/0/2}{0/1,0/2,1/3}%
            \hasse{0/0/0,1/1.2/0,2/0/1,3/0/2}{0/2,2/3}%
            \hasse{0/0/0,1/1/0,2/0/1,3/1/1}{0/2,0/3,1/3}%
            \hasse{0/0/0,1/1.5/0,2/0/1,3/.5/2}{0/2,1/3,2/3}%
            \hasse{0/0/0,1/0/1,2/0/2,3/0/3}{0/1,1/2,2/3}\\
  \end{tabular}
  \caption{The interval orders in degrees $1$ to $4$, grouped by
    molecular species.}
  \label{fig:molecular}
\end{figure}
\FloatBarrier

\section{Rigid interval orders}
\label{sec:rigid}

An interval order is \emph{rigid} if its only automorphism is the
identity. Two elements are \emph{indistinguishable} if they have the
same strict down-set and the same strict up-set. Dukes, Jel\'{\i}nek and
Kubitzke~\cite{DJK2011} prove the following as part of their Lemma~8.
We give a short proof from Corollary~\ref{cor:aut}.

\begin{lemma}[Dukes, Jel\'{\i}nek and Kubitzke~\cite{DJK2011}]\label{lem:entries}
  The nonempty entries of the composition matrix of an interval order
  are precisely its indistinguishability classes.
\end{lemma}

\begin{proof}
  The transposition exchanging $x$ and $y$ is an automorphism precisely
  when $x$ and $y$ are indistinguishable. By Corollary~\ref{cor:aut}, it
  is an automorphism precisely when $x$ and $y$ lie in the same entry.
\end{proof}

For instance, $3$ and $4$ are the only indistinguishable elements in
Example~\ref{ex:canon}. Both have strict down-set $\{1\}$ and strict
up-set $\{6\}$, and they form the entry $\{3,4\}$.

Corollary~\ref{cor:aut} shows that an interval order is rigid if and only
if every nonempty entry of its composition matrix is a singleton.

\begin{theorem}\label{thm:rigidspecies}
  Let $\cR\subseteq\I$ be the subspecies of rigid interval orders.
  Partitioning an interval order into indistinguishability classes gives
  a natural isomorphism
  \[
    \I=\cR\circ E_+.
  \]
\end{theorem}

\begin{proof}
  Fix a finite set $U$. An $(\cR\circ E_+)$-structure on $U$ is a
  partition $\beta$ of $U$ together with a rigid interval order on $\beta$.
  Let $P\in\I[U]$, and let $\beta$ be its partition into
  indistinguishability classes. Define a relation $<_Q$ on $\beta$ by
  \[
    B<_Q B'
    \quad\Longleftrightarrow\quad
    x<_P y
    \quad\text{for some $x\in B$ and $y\in B'$}.
  \]
  Since elements in the same block have the same strict down-sets and
  strict up-sets, ``some'' may be replaced by ``every''. Thus, $Q$ is
  well defined. Choosing one element from each block identifies $Q$ with
  an induced subposet of $P$, so $Q$ is an interval order. The strict
  down-set and up-set of $x\in B$ in $P$ are, respectively,
  \[
    \{y\in U:y<_P x\}=\bigcup_{A<_Q B}A,
    \qquad
    \{y\in U:x<_P y\}=\bigcup_{B<_Q A}A.
  \]
  Thus, if two blocks were indistinguishable in $Q$, their elements would
  be indistinguishable in $P$ and hence lie in the same block of $\beta$.
  The elements of $Q$ are therefore pairwise distinguishable. By
  Lemma~\ref{lem:entries}, every nonempty entry of its composition matrix
  is a singleton, so Corollary~\ref{cor:aut} shows that $Q$ is rigid.

  Conversely, let $\beta$ be a partition of $U$ and let $Q\in\cR[\beta]$.
  Define a relation on $U$ by $x<_P y$ if and only if $B<_Q B'$, where
  $B,B'\in\beta$ are the blocks containing $x,y$. Take an interval
  representation of $Q$ and give every $x\in B$ the interval of $B$.
  This represents $P$, so $P$ is an interval order.
  Since the elements of $Q$ are pairwise distinguishable, the same
  description of down-sets and up-sets shows that the indistinguishability
  classes of $P$ are precisely the blocks of $\beta$.
  The two constructions are inverse and commute with relabeling. Hence,
  they give the claimed natural isomorphism.
\end{proof}

\begin{corollary}\label{cor:rigidmol}
  With $w_m$ as above,
  \[
    \cR=\sum_{m\geq 0}w_m\,X^{m}.
  \]
\end{corollary}

\begin{proof}
  By Corollary~\ref{cor:aut}, an interval order is rigid if and only if
  every nonempty entry of its composition matrix is a singleton. Thus, in
  the molecular decomposition of Corollary~\ref{cor:molecular}, the rigid
  interval orders are indexed by the Fishburn matrices with entries zero
  or one. A matrix with $m$ ones contributes the molecule $X^m$, and there
  are $w_m$ such matrices by definition.
\end{proof}

Combining Theorem~\ref{thm:rigidspecies} and
Corollary~\ref{cor:rigidmol} gives
\[
  \I=\cR\circ E_+=\sum_{m\geq 0}w_m\,(E_+)^m.
\]
By Theorem~\ref{thm:bijective}, this recovers the fixed-point identity
\eqref{eq:fixballot} of Corollary~\ref{cor:fixedseries}. More explicitly,
the nonempty entries of $\kappa(A)$ are the indistinguishability classes
of $P(A)$. Replacing each nonempty entry $C_{ij}$ by $\{C_{ij}\}$ gives
the composition matrix of the rigid interval order on these classes.

Taking type generating series gives
\[
  \widetilde{\cR}(x)=\sum_{m\geq 0}w_m\,x^m,
  \qquad
  \widetilde{\I}(x)=\sum_{m\geq 0}w_m\Bigl(\frac{x}{1-x}\Bigr)^{\!m}.
\]
Thus, $w_m$ is the number of rigid interval orders on $m$ points up to
isomorphism. Substituting $x/(1+x)$ for $x$ in the second identity gives
\[
  \sum_{m\geq 0}w_m\,x^m=\widetilde{\I}\Bigl(\frac{x}{1+x}\Bigr).
\]
The enumerative content here is known. Brightwell and
Keller~\cite{BK2011} use the inflation of rigid orders in their asymptotic
analysis. Khamis~\cite{Khamis2012} and Dukes, Kitaev, Remmel and
Steingr\'{\i}msson~\cite{DKRS2011} obtain this series for
rigid orders, or equivalently for upper triangular $0$--$1$ matrices with
no zero row or column.

\section{Glaisher's T-numbers}
\label{sec:glaisher}

Zagier's~\cite{Zagier2001} asymptotic analysis of the Fishburn numbers
rests on the identity
\begin{equation}\label{eq:glaisher}
  e^{-t/24}\sum_{m\geq 0}\,\prod_{i=1}^{m}\bigl(1-e^{-it}\bigr)
  \;=\;
  \sum_{n\geq 0}\frac{T_n}{n!}\Bigl(\frac{t}{24}\Bigr)^{\!n},
\end{equation}
where $T_0,T_1,T_2,\dots$ are \emph{Glaisher's T-numbers}, listed as
A002439 in the OEIS~\cite{OEIS}. They begin $1, 23, 1681, 257543, 67637281,
\dots$ and are defined by
\[
  \sum_{n\geq 0}\frac{T_n}{(2n+1)!}\,x^{2n+1}
  \;=\;\frac{\sin 2x}{2\cos 3x}.
\]
We give a combinatorial interpretation of $T_n$ in terms of colored
interval orders.

After the rescaling $t=24x$, the left-hand side of \eqref{eq:glaisher}
is $e^{-x}\I(24x)$. Call an element of a poset \emph{isolated} if it is
comparable to no other element. We shall see that $e^{-x}\I(24x)$ is
the generating series of a species: the $24$-colored interval orders in
which no isolated element has color $24$.

For a positive integer $c$, a \emph{$c$-colored $F$-structure} on $U$
is an $F$-structure on $U$ together with a map $U\to[c]$. These form
the species $F(cX)$, the usual substitution of $cX$ into $F$. In
particular, $E(cX)= E^c$ and $F(cX)(x)=F(cx)$.

\begin{definition}
  Let $\cT$ be the species of $24$-colored interval orders in which no
  isolated element has color $24$.
\end{definition}

\begin{theorem}\label{thm:glaisher}
  There is a natural isomorphism $\I(24X)=E\cdot\cT$ and, consequently,
  \[
    \cT\;=\;E^{-1}\I(24X)
  \]
  in $\Virt$. Moreover, $|\cT[n]|=T_n$.
\end{theorem}

\begin{proof}
  A $24$-colored interval order is uniquely the disjoint union of its
  isolated elements of color $24$ and a $\cT$-structure. This decomposition
  is natural and proves $\I(24X)=E\cdot\cT$. Taking generating series gives
  $\cT(x)=e^{-x}\I(24x)$. Comparing with \eqref{eq:glaisher} at $t=24x$
  yields $|\cT[n]|=T_n$.
\end{proof}

This interpretation by colored interval orders seems to be new. The
OEIS~\cite{OEIS} records the generating function obtained from
\eqref{eq:glaisher} by setting $t=24x$, but gives no combinatorial
interpretation.
Hoffman~\cite{Hoffman1999} shows that $2T_n$ counts the
$ER_{2n+1}$-snakes, a class of $3$-signed alternating permutations.

From $\cT=E^{-1}\I(24X)$ we can also calculate the cycle
index series of $\cT$.

\begin{corollary}
  With $q=x_1+x_2/2+\cdots$ as before, the
  cycle index series of $\cT$ is
  \[
    Z_{\cT}
    =e^{-q}\sum_{m\geq 0}\prod_{i=1}^{m}\bigl(1-e^{-24iq}\bigr)
    =\cT(q).
  \]
  Here $\cT(x)=\sum_{n\geq 0}T_nx^n/n!$ as above. Its type generating
  series is
  \[
    \widetilde{\cT}(x)=(1-x)\sum_{m\geq 0}\prod_{i=1}^{m}
    \bigl(1-(1-x)^{24i}\bigr).
  \]
\end{corollary}


The coefficients of $\widetilde{\cT}(x)$ begin
\[
  1,23,852,43772,2883382,231864234,22020936532,\dots
\]
and this sequence is not in the OEIS~\cite{OEIS}.

\section*{Acknowledgments}

Claude (Anthropic) and GPT (OpenAI) were used during the
preparation of this manuscript. This included exploring and checking
conjectures through programming, formulating propositions, developing
proofs, and editing. The author, however, takes full responsibility for
the content of the manuscript.

\end{document}